\documentclass[11pt, english]{amsart}
\usepackage[utf8]{inputenc}

\usepackage{xcolor}
\usepackage{rotating}
\usepackage{pifont}
\usepackage{wasysym}

\usepackage{geometry}
\usepackage{fdsymbol}

\theoremstyle{plain}
\newtheorem{theorem}{Theorem}
\newtheorem*{conjecture*}{Conjecture}
\newtheorem{prop}[theorem]{Proposition}
\newtheorem{lemma}[theorem]{Lemma}
\newtheorem{coro}[theorem]{Corollary}
\newtheorem{remark}[theorem]{Remark}

\newcommand{\coeur}{{\varheartsuit}}

\def\CC{{\mathbb{C}}}

\def\PP{{\mathbb{P}}}
\def\QQ{{\mathbb{Q}}}\def\ZZ{{\mathbb{Z}}}

\def\lra{{\longrightarrow}}

\def\ev{\mathrm{ev}}
\def\Hdg{\mathrm{Hdg}}

\DeclareMathOperator{\pt}{pt}

\begin{document} 

\title[Quantum cohomology of Verra fourfolds]{Quantum cohomology and birational geometry \\ of Verra fourfolds}
\author{Vladimiro Benedetti, Jérémy Guéré, Laurent Manivel, Nicolas Perrin}
\date{\today}

\address{
Universit\'e C\^ote d'Azur, CNRS, Laboratoire J.-A. Dieudonn\'e, Parc Valrose, F-06108 Nice Cedex 2, France} 
\email{vladimiro.benedetti@univ-cotedazur.fr}

\address{
Univ. Grenoble Alpes, CNRS, IF, 38000 Grenoble, France}

\email{jeremy.guere@univ-grenoble-alpes.fr}

\address{
Institut de Math\'ematiques de Marseille, 
Aix-Marseille University, CNRS, I2M, UMR 7373, Marseille, France}
\email{laurent.manivel@math.cnrs.fr}

\address{
Centre de Math\'ematiques Laurent Schwartz (CMLS), CNRS, \'Ecole polytechnique,
Institut Polytechnique de Paris, F-91120 Palaiseau, France}
\email{nicolas.perrin.cmls@polytechnique.edu}

\begin{abstract}
We compute the small quantum cohomology ring of a Verra fourfold.
Using the theory of atoms recently developped by Katzarkov--Kontsevich--Pantev--Yu, and building on recent papers of the authors, we deduce that a Verra fourfold is never birational to a very general cubic fourfold, nor to a very general Gushel--Mukai fourfold, whereas it was previously known that a general Verra fourfold is birational to a general nodal Gushel--Mukai fourfold.
More precisely, we show that for every smooth cubic fourfold or smooth  Gushel--Mukai fourfold that is birational to some Verra fourfold, the primitive cohomology is isomorphic, as a rational Hodge structure, to the 
middle cohomology of some projective K3 surface.
\end{abstract} 

\maketitle

\section{Introduction}

The theory of Hodge atoms developed in \cite{kkpy} is a powerful tool to prove irrationality of fourfolds, at least under certain positive circumstances like having a nonzero Hodge number $h^{3,1}$. This happens for cubic fourfolds: irrationality of the very general cubic fourfold was the most astonishing achievement of \cite{kkpy}.
Many varieties with similar properties were  listed in \cite{bfmt}, starting with Gushel-Mukai varieties. The irrationality of a very  general Gushel-Mukai variety was established in \cite{bmp}. The next natural candidates are Verra fourfolds, whose Picard number is two in general; a general Verra fourfold is birational to a general nodal Gushel-Mukai fourfold by \cite[Corollary 5.2]{bkk} (see also \cite{grz}).




A Verra fourfold $X$ has two structures of quadric surface fibration over $\PP^2$, hence two associated  K3 surfaces of degree two $\Sigma_1$ and $\Sigma_2$. These surfaces are not isomorphic in general \cite{verra, kkm}. It follows from the work of Laszlo \cite{laszlo} that, over $\QQ$ and when the discriminant sextic curve is smooth, the primitive cohomology of  $X$ coincides with the primitive cohomology of both K3 surfaces, which implies that $\Sigma_1$ and $\Sigma_2$ are twisted derived equivalent \cite{huybrechts}.

\medskip
In order to prove our main theorem, we compute the small quantum cohomology of a Verra fourfold. We deduce, using the language of \cite{kkpy}, that a Verra fourfold has a Hodge atom with the same invariants as the Hodge atom of a K3 surface. The indecomposability of this atom follows from the fact that for a rational Verra fourfold, it has to coincide with the atom of a K3 surface appearing as a blow-up center in a weak factorization of a birational map to projective space. 

Such an atom cannot appear in a very general cubic or Gushel-Mukai fourfold. Following a similar argument to those that were used in \cite{kkpy, jg, bmp} to show that a very general cubic or Gushel-Mukai fourfold is not rational, we prove the following theorem:

\begin{theorem}
Let $X$ be a Verra fourfold.
\begin{enumerate}
    \item $X$ is not birational to a very general cubic fourfold nor to a very general Gushel-Mukai fourfold.
\item More precisely, if $Y$ is a cubic fourfold or a Gushel-Mukai fourfold birational to $X$, then there exists a projective $K3$ surface $\Sigma$ and an isomorphism of rational Hodge structures
$$H^4(Y,\mathbb{Q})_\mathrm{van} \simeq H^2(\Sigma,\mathbb{Q})(-1).$$
In particular $Y$ is not Hodge general. 
\end{enumerate}
\end{theorem}

\begin{remark}
$H^4(Y,\mathbb{Q})_\mathrm{\mathrm{van}}$ is the vanishing cohomology of the embedding, respectively, $Y\subset \PP^5$ for cubics and $Y\subset G(2,5)$ for Gushel-Mukai. Of course, for cubics this corresponds to the primitive cohomology.
\end{remark}

\subsection*{Comparaison to Fay's work}
Interestingly, Verra fourfolds in this argument play the same role as projective spaces. In fact, the criteria $\clubsuit$
and $\coeur$ developed in \cite{jg,kkpy} are not sufficient to decide whether a very general Verra fourfold is rational or not, a question that to date remains open.
Specifically, this is because of the special K3 geometry that Verra fourfolds have. As written above, even for a very general Verra fourfold $X$, the primitive cohomology of $X$ is isomorphic as a rational Hodge structure to the primitive cohomology of its two K3 surfaces, so that criteria in \cite{jg,kkpy} remain inconclusive regarding irrationality of $X$.
This is the reason why we look instead at the birational comparaison between Verra fourfolds and cubics/Gushel--Mukai fourfolds.

In her paper \cite{fay}, Aideen Fay looks at another interesting question, in the spirit of equivariant atoms developped by \cite{kkpy,ckk}.
Specifically, she considers \textit{symmetric} Verra fourfolds $X$ (under swapping of the $\mathbb{P}^2 \times \mathbb{P}^2$ factors) together with its involution $\epsilon$ and she addresses the question of whether there exists a regular involution $\iota$ of $\mathbb{P}^4$ and a birational map from $X$ to $\mathbb{P}^4$ commuting with these involutions.
When this fails, the pair $(X,\epsilon)$ is called $\ZZ_2$-irrational.
Her main theorem is the following:
the very general symmetric Verra fourfold is $\ZZ_2$-irrational.

Note also that computations of the small quantum product are independently obtained by \cite{fay} using different techniques (namely, quantum periods), confirming the calculations in this paper.

\bigskip\noindent {\it Acknowledgements}. We thank Aideen Fay and Tom Coates for useful discussions.
V.B. and J.G. would like to thank Ludmil Katzarkov and IMSA for the opportunity to present their results at the \textit{Hodge theory, Birational Geometry, and Atoms} conference, where many discussions about this paper occurred.
V.B., L.M. and N.P.  were supported by the project FanoHK ANR-20-CE40-0023.


\section{Verra fourfolds}

\subsection{Ambient cohomology}

A {\it Verra fourfold} $X$ is defined as a double covering $\pi: X\lra\PP^2\times\PP^2$ branched over a divisor of bidegree $(2,2)$. Like Gushel-Mukai fourfolds, Verra fourfolds have index two, but their Picard number is two.  
The  Hodge diamond is as follows (note the slight discrepancy with \cite{laterv}):
\small
$$\begin{matrix}
&&&&1&&&&\\
&&&0&&0&&&\\
&&0&&2&&0&& \\
&0&&0&&0&&0&\\
0&&1&&22&&1&&0 \\
&0&&0&&0&&0&\\
&&0&&2&&0&& \\
&&&0&&0&&&\\
&&&&1&&&&
\end{matrix}$$
\normalsize

\medskip
Let $p_1,p_2$ be the two projections $X\lra \PP^2$. Each of these realizes $X$ as a quadric surface bundle, whose discriminant locus is a plane sextic curve. The double cover of $\PP^2$ branched over such a sextic, when smooth, is a degree two K3 surface. According to \cite{kkm}, the two projections yield two  K3 surfaces $\Sigma_1$ and $\Sigma_2$ which are in general non-isomorphic, but twisted derived equivalent, in the sense that they come with Brauer classes $\beta_1$ and  $\beta_2$, respectively, such that the twisted derived categories $D^b(\Sigma_1,\beta_1)$ and $D^b(\Sigma_2,\beta_2)$ are equivalent.

\smallskip
By a slight  abuse of terminology, we define the ambient cohomology of $X$ to be the pull-back of the cohomology of $\PP^2\times\PP^2$ and denote it by $H^\bullet(X)^{\mathrm{amb}}$. Denote by $H^\bullet(X)^{\mathrm{Hod}}$ the space of Hodge classes and let $h_1$ and $h_2$ be the hyperplane classes pull-backed by the two projections $p_1$ and $p_2$.

\begin{lemma}
\label{lem_Hodgeclasses_verygeneral}
For $X$ a very general Verra fourfold, we have 
\begin{equation}\label{general}
H^\bullet(X)^{\mathrm{Hod}}=H^\bullet(X)^{\mathrm{amb}}=\QQ\langle 1, h_1, h_2, h_1^2, h_1h_2, h_2^2, h_1^2h_2, h_1h_2^2, h_1^2h_2^2\rangle .
\end{equation}
\end{lemma}

Since $h_1^2h_2^2=2\pt$,  this basis of $H^\bullet(X)^{\mathrm{Hod}}$ is Poincar\'e self-dual, up to a factor two. When (\ref{general}) holds, we say that $X$ is {\it Hodge general}. 

\proof 
It follows from  \cite{laszlo} that the transcendental lattice in $H^4(X,\QQ)$
is isomorphic to the transcendental lattice of either $\Sigma_1$ or $\Sigma_2$. When $X$ varies, these K3 surfaces move in a $19$-dimensional family. Thus for $X$ very general they must have Picard number $1$ and transcendantal lattice of rank $21$. Since $b_4(X)=24$, this ensures that the Hodge classes in $H^4(X,\QQ)$ span a three-dimensional space, which has to be $\QQ \langle h_1^2, h_1h_2, h_2^2\rangle$. \qed

\medskip
Notice that the ambient cohomology $H^\bullet(X)^{\mathrm{amb}}$ is the pullback to $X$ of the cohomology of $\PP^2\times \PP^2$. In other words, it is the invariant cohomology under the action of the involution defining the double cover $X\to \PP^2\times \PP^2$. The anti-invariant part coincides with the $\cup$-orthogonal complement of $H^\bullet(X)^{\mathrm{amb}}$ inside $H^\bullet(X)$. Clearly, this cohomology is concentrated in degree four and its dimension is $21$. Following \cite[Section II.2.2]{laszlo}, we refer to it as the primitive part of the cohomology and denote it  $H^\bullet(X)^{\mathrm{prim}}$.

\subsection{Lines and conics on Verra fourfolds}
In order to determine the whole quantum multiplication on the ambient cohomology, we only need to compute the quantum multiplication by $h_1$ or $h_2$. 
For this, as in the case of Gushel-Mukai varieties, we study ``lines" and ``conics" on a general Verra fourfold. 

By a ``line" we mean a curve of class $h_1^\vee$ (a $(1,0)$-line), or $h_2^\vee$ (a $(0,1)$-line). A $(1,0)$-line is contained in a fiber of the second projection $p_2$ from $X$ to $\PP^2$. This fiber is a double cover of $\PP^2$ branched over a conic, hence a quadratic surface. If the conic is smooth, the preimage of each tangent line is a pair of $(1,0)$-lines, whose set is therefore parametrized by two copies of the conic. If the conic is singular (but reduced for $X$ general), with a singular point $p$, the $(1,0)$-lines in the fibers are obtained in pairs as preimages of lines through $p$, and we get twice the same $(1,0)$-line when this line through $p$ is a component of the singular conic; so the set of $(1,0)$-lines in the fiber is parametrized by a double cover of $\PP^1$ branched at two points, which is again $\PP^1$. This shows that the set of $(1,0)$-lines is parametrized by a $\PP^1$-fibration over a double cover of $\PP^2$ branched over a discriminant locus which is a smooth sextic curve. We deduce:

\begin{prop} Let $X$ be a general Verra fourfold. The Hilbert scheme $F_{(1,0)}(X)$ of $(1,0)$-lines on $X$ is a smooth threefold, with a $\PP^1$-fibration over a K3 surface. 
\end{prop}

\medskip
The case of conics is of course more complicated. There are three types of conics: $(2,0)$-conics contracted by $p_2$, $(0,2)$-conics contracted by $p_1$, and $(1,1)$-conics mapped to lines by both $p_1$ and $p_2$. The following result is from \cite[Theorem 0.2]{ikkr}.

\begin{prop}
The Hilbert scheme $F_{(1,1)}(X)$ of $(1,1)$-conics on $X$
is a smooth fivefold, with a $\PP^1$-fibration over a smooth double EPW quartic. 
\end{prop}

\subsection{Some rational Verra fourfolds} 
The general Verra fourfold is birational to a general nodal Gushel-Mukai fourfold by \cite{bkk} (see \cite[Theorem 4.5]{dim0} for dimension three). Note that along with cubic fourfolds blown up along a plane, Verra fourfolds provide one of the types of quadric surface bundles  to which \cite[Corollary 2]{schreieder} does not apply.
\medskip

 We can define a period map for Verra fourfolds. The $\ZZ$-module $H^4(X,\ZZ)$ is endowed with an intersection form $Q_X$. Denote by $(\Lambda,Q)$ the abstract lattice given by $(H^\bullet(X,\ZZ),Q_X)$. The period domain for Verra fourfolds is the quotient
$$
    \mathcal{P} := \left\{ \begin{aligned} \relax [w]\in 
    \PP(\Lambda\otimes_{\ZZ} \mathbb{C} ) \mid Q(w,\overline{w})>0,\quad  Q(w,w)=0, \\
    Q(w,h)=0 \quad \forall h\in H^4(X)^{\mathrm{amb}}\otimes_\QQ \CC \end{aligned} \right\} \Big/ \Gamma,
$$
where $\Gamma$ is the monodromy group of Verra fourfolds. The period map for Verra fourfolds is then defined as
$$
 \mathcal{M}\to \mathcal{P}\quad , \quad [X] \mapsto [H^{3,1}(X,\CC)],
$$
where $\mathcal{M}$ is the moduli space of Verra fourfolds (see \cite[Subsection 0.3]{ikkr} for more details on this moduli space).

\begin{prop}
\label{prop_rational_Verra}
There exists a Heegner divisor $D$ in the period domain $\mathcal{P}$ such that each smooth Verra fourfold whose image via the period map is contained in $D$ is rational. 
\end{prop}

\begin{proof}
For $X$ a Verra fourfold, consider the quadric surface bundle structure given e.g. by the projection $p_1: X\lra\PP^2$. Suppose that a fiber splits into the union of two disjoint planes, which will happen when the discriminant sextic curve in $\PP^2$ admits a simple singularity. Then each of these planes 
gives a section of the other projection $p_2: X\lra\PP^2$. It follows that $X$ is rational (see \cite[Proposition 6.3]{ckkm} and \cite[Corollary 6.4]{ckkm}).

Denote by $T\in \Lambda \subset  H^2(X,\ZZ)$ the class of this section, and consider $$D:=\{[w]\in \mathcal{P}\mid Q(w,T)=0 \}\subset \mathcal{P}. $$
Then $D \subsetneq \mathcal{P}$ by Lemma \ref{lem_Hodgeclasses_verygeneral}, and $D$ is a non-empty divisor in $\mathcal{P}$.

Now, for any Verra fourfold $X'$ whose image by the period map is contained in $D$, the class $T$ corresponds to an integral class $T_X$ in the cohomology of $X$; since $Q_X(T_X,h_1^2)=Q(T,h_1^2)=1$, \cite[Proposition 2.1]{hassett} implies that $X'$ is rational.
\end{proof}

\begin{remark}
Other examples of rational Verra fourfolds are constructed in \cite{kkm} (see section 4 and Proposition 4.1) and in \cite{grz}.
It would be interesting to find other Heegner divisors in the period domain, parametrizing rational Verra fourfolds. Since the general Verra fourfold is birational to a nodal Gushel-Mukai fourfold, one could try to analyse the intersection of the Heegner divisor associated to these nodal Gushel-Mukai's, with other  Heegner divisors parametrizing rational Gushel-Mukai fourfolds, and describe geometrically the corresponding special Verra fourfolds. 
\end{remark}

\section{The small quantum cohomology ring} 

\subsection{Quantum multiplication by a hyperplane class}
We start with the degree two products. By symmetry, they must a priori be of the form 
$$\begin{array}{rcl}
h_1\star h_1 & = & h_1^2+Aq_1+Bq_2, \\
h_1\star h_2 & = & h_1h_2+Cq_1+Cq_2, \\
h_2\star h_2 & = & h_2^2+Bq_1+Aq_2.
\end{array}$$

The Gromov-Witten invariants
$B=\langle h_1, h_1,\pt\rangle_{(0,1)}$ and 
$C=\langle h_1, h_2,\pt\rangle_{(0,1)}$ are clearly zero by the Divisor Axiom. In fact any invariant of the form 
$\langle h_1, \bullet,\bullet\rangle_{(0,1)}$ or 
$\langle h_2, \bullet,\bullet\rangle_{(1,0)}$ must be zero for the same reason.

\begin{lemma}
$A=2$.
\end{lemma}

\proof The Gromov-Witten invariant $A=\langle h_1, h_1,\pt\rangle _{(1,0)}$
counts $(1,0)$-lines through a general point $x$ of $X$. Such a curve is contained in the fiber of $p_2$ over $p_2(x)$, which is a double cover of $\PP^2$ branched over a smooth conic since $p_2(x)$ is general. We have seen that the curves we are looking for are contained in the pre-images of tangent lines to the conic; there are two such tangents passing through $p_1(x)$, whose preimages give four $(1,0)$-lines; but only half of them pass through $x$, hence the claim.\qed 

\medskip We conclude that 
$$
h_1\star h_1 = h_1^2+2q_1, \qquad h_1\star h_2 = h_1h_2, \qquad 
h_2\star h_2 = h_2^2+2q_2.
$$

Now we turn to degree three products. Let us start with 
$$\begin{array}{rcl}
h_1^2\star h_1 & = & (Dq_1+Eq_2)h_1+(Fq_1+Gq_2)h_2, \\
h_2^2\star h_1 & = & h_1h_2^2+(Hq_1+Iq_2)h_1+(Jq_1+Kq_2)h_2.
\end{array}$$

The Divisor Axiom yields $E=G=I=K=0$. 

\begin{lemma}
$D=2$, $H=J=0$.
\end{lemma}

\proof 
In order to compute the Gromov-Witten invariant $D=\langle h_1^2,  h_1^\vee\rangle _{(1,0)}$, 
we can represent the class $h_1^2$ by the pre-image in $X$ of $\{p\}\times \PP^2$ for a generic point $p\in\PP^2$, and  the class $2h_1^\vee$ by the pre-image  of $\ell\times \{q\}$ for a generic line $\ell\subset\PP^2$ and a generic point $q\in\PP^2$. 
Then we need to count $(1,0)$-lines contracting to $q$, that we know to be pre-images of tangent lines to a smooth conic in $\PP^2$. These lines have to meet $\ell$, and to pass through $p$; there are two such lines, yielding four $(1,0)$-lines, hence $D=2$. 

In the same way, to compute the Gromov-Witten invariants $H=\langle h_2^2,  h_1^\vee\rangle _{(1,0)}$ and $J=\langle h_2^2,  h_2^\vee\rangle _{(1,0)}$,
we represent the class $h_2^2$ by the pre-image in $X$ of $\PP^2\times \{p\} $ for a generic point $p\in\PP^2$, and  the class $2h_1^\vee$ (resp. $2h_2^\vee$) by the pre-image  of $\ell\times \{q\}$ (resp. $\{q\}\times \ell$) for a generic line $\ell\subset\PP^2$ and a generic point $q\in\PP^2$. Then we need to count $(1,0)$-lines contracted to $p$, and also to  $q$ 
(resp. some point of $\ell$); there are no such lines, so $H=J=0$. \qed 

\begin{lemma}
$F=4$.
\end{lemma}

\begin{proof}
In order to compute $F=
\langle h_1^2, h_2^\vee\rangle _{(1,0)}$, we represent the class $h_1^2$ by a fiber $p_1^{-1}(p)$ for a generic point $p\in\PP^2$, and the class $2h_2^\vee=h_1^2h_2$   by the preimage of $\{q\}\times \ell\subset \PP^2\times\PP^2$, for a generic point $q$ and a generic line $\ell$. 

Now, a line of bidegree $(1,0)$ is contracted by $p_2$,
and the lines we count must be contained in a fiber $p_2^{-1}(r)$, for some $r\in \ell$. This fiber is a double cover of $\PP^2$, branched over a conic $c_r$. We are looking for lines $\delta$ in this cover, whose image in $\PP^2$ pass through both $q$ and $q$. This means that $\overline{pq}$ is tangent to $c_r$, and $\delta$ is one of the two components of its preimage in $X$. Recall that if we fix homogeneous coordinates on $\PP^2$, the matrix of $c_r$ has quadratic entries. Being tangent to a fixed line is equivalent to the vanishing of a suitable $2\times 2$ minor of this matrix, hence it is a quartic condition. In other words, there are in general four points $r$ on $\ell$ for which $\overline{pq}$ is tangent to $c_r$, each such tangent pulling-back to a pair of $(1,0)$-lines in $X$. Hence 
$$\langle h_1^2, h_1^2h_2\rangle _{(1,0)}= 8 \qquad\mathrm{and}\qquad F=\langle h_1^2, h_2^\vee\rangle _{(1,0)}=4. \qedhere $$
\end{proof}

\smallskip
Once we know $h_1^2\star h_1$ and $h_2^2\star h_1$, by symmetry we also know $h_2^2\star h_2$ and $h_1^2\star h_2$. Moreover we can deduce $h_1h_2\star h_1$ and $h_1h_2\star h_2$ from the associativity of the quantum product, as follows: 
$$h_1h_2\star h_1=(h_1\star h_2)\star h_1=(h_1\star h_1)\star h_2=
(h_1^2+2q_1)\star h_2=h_1^2\star h_2+2q_1h_2.$$
We finally get the following degree three products:
$$\begin{array}{rclrcl}
h_1^2\star h_1 & = & 2q_1h_1+4q_1h_2, \qquad & 
h_1^2\star h_2  &=  &h_1^2h_2, \\ 
h_1h_2\star h_1 & = & h_1^2h_2+2q_1h_2, \qquad &
h_1h_2\star h_2 & = & h_1h_2^2+2q_2h_1,\\
h_2^2\star h_1  &=  &h_1h_2^2,\qquad & 
h_2^2\star h_2 & = & 2q_2h_2+4q_2h_1.
\end{array}$$

Now we turn to degree four products, which involve degree two Gromov-Witten invariants. We can compute directly
$$\begin{array}{ll}
h_1^2h_2\star h_1 & =(h_1^2\star h_2)\star h_1=(h_1^2\star h_1)\star h_2 \\
& = h_2+4q_1h_2\star h_2= 2q_1h_1h_2+4q_1h_2^2+8q_1q_2.
\end{array}$$
On the other hand, we have 
$$h_1^2h_2\star h_2=(h_1^2\star h_2)\star h_2=h_1^2\star (h_2\star h_2)=h_1^2\star (h_2^2+2q_2),$$ and symmetrically we deduce that
$$h_1^2\star h_2^2=h_1^2h_2\star h_2-2q_2h_1^2=h_1h_2^2\star h_1-2q_1h_2^2.$$
By homogeneity, we can a priori write the product $h_1^2\star h_2^2$ as
$$h_1^2h_2^2+(Lq_1+Mq_2)h_1^2+N(q_1+q_2)h_1h_2+(Mq_1+Lq_2)h_2^2+S(q_1^2+q_2^2)+Tq_1q_2.$$

\begin{lemma}
$L=M=N=0$.
\end{lemma}

\proof That $2L=\langle h_1^2,h_2^2,h_2^2\rangle _{(1,0)}=0$ simply follows from the fact that this Gromov-Witten invariant counts curves that should be contracted by $p_2$ to two distinct points. Similarly $N=0$
since $\langle h_1^2,h_2^2,h_1h_2\rangle _{(1,0)}$ counts curves that should be contracted by $p_2$ to a fixed point and a general line. Finally,  
$M=0$  as well, since  $\langle h_1^2,h_2^2,h_1^2\rangle _{(1,0)}$ counts curves coming from tangents to a given conic, passing through two general points of $\PP^2$ - and there is no such tangent. \qed

\medskip
There remains to compute the quadratic terms. 

\begin{lemma}
$S=0$ and $T=8$. 
\end{lemma}

\proof Clearly $S=0$, since $\langle h_1^2,h_2^2,\pt \rangle _{(2,0)}$  counts curves that, once again, should be contracted by $p_2$ to two distinct points. 

In order to compute $T$, we note that 
$$h_1^2h_2^2\star h_1=(h_1^2\star h_2^2-Tq_1q_2)\star h_1=2q_1h_1h_2^2+8q_1q_2h_2+16 q_1q_2h_1-Tq_1q_2h_1.$$
In particular $16-T=\langle h_1^2h_2^2, h_1, h_1^\vee\rangle _{(1,1)}$. By the Divisor Axiom this is also $\langle h_1^2h_2^2, h_1^\vee\rangle _{(1,1)}=\langle h_1^2h_2^2, h_2, 
h_1^\vee\rangle _{(1,1)}$, or $\langle h_1^2h_2^2, h_1, 
h_2^\vee\rangle _{(1,1)}$ by symmetry. But the coefficient of $q_1q_2h_2$ in $h_1^2h_2^2\star h_1$ is $8$, hence finally $T=16-8=8$. \qed

\medskip
We deduce all the missing products:
$$\begin{array}{rcl}
h_1^2h_2\star h_1 & = & 2q_1h_1h_2+4q_1h_2^2+8q_1q_2,\\
h_1h_2^2\star h_1 & = & h_1^2h_2^2+2q_1h_2^2+8q_1q_2,\\
h_1^2h_2^2\star h_1& = & 2q_1h_1h_2^2+8q_1q_2h_1+8q_1q_2h_2, 
\end{array}$$

$$\begin{array}{rcl}
h_1^2h_2\star h_2  &=  &h_1^2h_2^2+2q_2h_1^2+8q_1q_2, \\
h_1h_2^2\star h_2 & = & 2q_2h_1h_2+4q_2h_1^2+8q_1q_2,\\
h_1^2h_2^2\star h_2& = & 2q_2h_1^2h_2+8q_1q_2h_1+8q_1q_2h_2.
\end{array}$$

\begin{prop} 
\label{prop:mult-div}The quantum product by $a_1h_1+a_2h_2$ is  given by the following matrix:
$$M_{a_1h_1+a_2h_2}=\small\begin{pmatrix}
0&2b_1&2b_2&0&0&0&c&c&0 \\
a_1&0&0&2b_1&2b_2&4b_2&0&0&c \\
a_2&0&0&4b_1&2b_1&2b_2&0&0&c \\
0&a_1&0&0&0&0&2b_2&4b_2&0 \\
0&a_2&a_1&0&0&0&2b_1&2b_2&0 \\
0&0&a_2&0&0&0&4b_1&2b_1&0 \\
0&0&0&a_2&a_1&0&0&0&2b_2 \\
0&0&0&0&a_2&a_1&0&0&2b_1 \\
0&0&0&0&0&0&a_2&a_1&0
\end{pmatrix},$$
where $b_1=a_1q_1$, $b_2=a_2q_2$, $c=8(a_1+a_2)q_1q_2$
\end{prop}

\normalsize

\subsection{Quantum multiplication by the anticanonical class} 

As a special case of Proposition \ref{prop:mult-div}, the matrix of the quantum product by $c_1(X)=2(h_1+h_2)$ is 
$$M_{c_1(X)}=\small\begin{pmatrix}
0&4q_1&4q_2&0&0&0&32q_1q_2&32q_1q_2&0 \\
2&0&0&4q_1&4q_2&8q_2&0&0&32q_1q_2 \\
2&0&0&8q_1&4q_1&4q_2&0&0&32q_1q_2 \\
0&2&0&0&0&0&4q_2&8q_2&0 \\
0&2&2&0&0&0&4q_1&4q_2&0 \\
0&0&2&0&0&0&8q_1&4q_1&0 \\
0&0&0&2&2&0&0&0&4q_2 \\
0&0&0&0&2&2&0&0&4q_1 \\
0&0&0&0&0&0&2&2&0
\end{pmatrix}$$

\normalsize

\begin{prop}
\label{comput}
For generic $q_1,q_2$, the multiplication by the canonical class in the ambient cohomology has six nonzero simple eigenvalues, and a three-dimensional kernel   
$E_0^{\mathrm{amb}}= \langle \alpha, \beta, \gamma \rangle$ where
$$
\alpha:=-h_1^2+h_1h_2-h_2^2+2(q_1+q_2),\quad \beta:=h_1h_2(h_2-h_1)+2(q_1h_2-q_2h_1),
$$
$$
\gamma:=h_1^2h_2^2-2(q_1h_2^2+q_2h_1^2+2q_1q_2).
$$
\end{prop}

\begin{prop}\label{zero-antiinvariant}
The multiplication by the canonical class restricted to the primitive cohomology $H^\bullet(X)^{\mathrm{prim}}$ is zero.
\end{prop}

\begin{proof}
By \cite[Theorem 1.6]{LeePirola} the monodromy of the family of Verra fourfolds (i.e. double covers of  $\PP^2\times \PP^2$ ramified over a $(2,2)$-divisor) acts on $H^\bullet(X)^{\mathrm{prim}}$ irreducibly. By the invariance of Gromov-Witten-invariants under deformation, we deduce that all Gromov-Witten-invariants with only one anti-invariant insertion vanish. Now, let $\gamma\in H^\bullet(X)^{\mathrm{prim}}$; then
$$
(2h_1+2h_2) \star \gamma= (2h_1+2h_2) \cup \gamma +\sum_{d_1,d_2\geq 1} \sum_\beta \langle (2h_1+2h_2), \gamma, \beta^\vee \rangle_{d_1,d_2}q_1^{d_1}q_2^{d_2}\beta.
$$
Notice that, for any $d_1,d_2$, since $\deg(\gamma)=4$, the degree of $\beta$ (or $\beta^\vee$) is not equal to four. Therefore $\beta^\vee$ is invariant, and $\langle (2h_1+2h_2), \gamma, \beta^\vee \rangle_{d_1,d_2}$ has only one anti-invariant insertion, hence it vanishes. 
\end{proof}

\subsection{Quantum products by the other ambient classes}
Since $h_1$ and $h_2$ generate the full ambient cohomology algebra (at least over $\QQ$), we can easily deduce, from the quantum multiplication rules by $h_1$ and $h_2$, the full quantum multiplication table for the ambient cohomology. For convenience we record the matrices $M_\tau$ of quantum multiplication by a monomial class $\tau$ of degree bigger than one. We use the notation $q_0=q_1+q_2$.

$$M_{h_1^2}=\small
\begin{pmatrix}
0&0&0&4q_1^2&8q_1q_2&8q_1q_2&0&0&32q_1^2q_2 \\
0&2q_1&0&0&0&0&8q_1q_2&16q_1q_2&0 \\
0&4q_1&0&0&0&0&4q_1^2&8q_1q_2&0 \\
1&0&0&0&0&0&0&0&8q_1q_2 \\
0&0&0&4q_1&2q_1&0&0&0&8q_1q_2 \\
0&0&0&0&4q_1&0&0&0&4q_1^2 \\
0&0&1&0&0&0&0&0&0 \\
0&0&0&0&0&0&4q_1&2q_1&0 \\
0&0&0&0&0&1&0&0&0 
\end{pmatrix}$$
\normalsize

$$M_{h_1h_2}=\small\begin{pmatrix}
0&0&0&8q_1q_2&12q_1q_2&8q_1q_2&0&0&16q_1q_2q_0 \\
0&0&2q_2&0&0&0&20q_1q_2&16q_1q_2&0 \\
0&2q_1&0&0&0&0&16q_1q_2&20q_1q_2&0 \\
0&0&0&0&2q_2&4q_2&0&0&8q_1q_2 \\
1&0&0&2q_1&0&2q_2&0&0&12q_1q_2 \\
0&0&0&4q_1&2q_1&0&0&0&8q_1q_2 \\
0&1&0&0&0&0&0&2q_2&0 \\
0&0&1&0&0&0&2q_1&0&0 \\
0&0&0&0&1&0&0&0&0 
\end{pmatrix}$$

\normalsize

$$M_{h_2^2}=\small\begin{pmatrix}
0&0&0&8q_1q_2&8q_1q_2&4q_2^2&0&0&32q_1q_2^2 \\
0&0&4q_2&0&0&0&8q_1q_2&4q_2^2&0 \\
0&0&2q_2&0&0&0&16q_1q_2&8q_1q_2&0 \\
0&0&0&0&4q_2&0&0&0&4q_2^2 \\
0&0&0&0&2q_2&4q_2&0&0&8q_1q_2 \\
1&0&0&0&0&0&0&0&8q_1q_2 \\
0&0&0&0&0&0&2q_2&4q_2&0 \\
0&1&0&0&0&0&0&0&0 \\
0&0&0&1&0&0&0&0&0 
\end{pmatrix}$$

\normalsize

$$M_{h_1^2h_2}=\small\begin{pmatrix}
0&8q_1q_2&8q_1q_2&0&0&0&40q_1^2q_2&16q_1q_2q_0&0 \\
0&0&0&8q_1q_2&20q_1q_2&8q_1q_2&0&0&16q_1q_2q_0 \\
0&0&0&4q_1^2&16q_1q_2&16q_1q_2&0&0&40q_1^2q_2 \\
0&0&2q_2&0&0&0&16q_1q_2&8q_1q_2&0 \\
0&2q_1&0&0&0&0&16q_1q_2&20q_1q_2&0 \\
0&4q_1&0&0&0&0&4q_1^2&8q_1q_2&0 \\
1&0&0&0&0&2q_2&0&0&8q_1q_2 \\
0&0&0&4q_1&2q_1&0&0&0&8q_1q_2 \\
0&0&1&0&0&0&0&0&0 
\end{pmatrix}$$

\normalsize

$$M_{h_1h_2^2}=\small\begin{pmatrix}
0&8q_1q_2&8q_1q_2&0&0&0&16q_1q_2q_0&40q_1q_2^2&0 \\
0&0&0& 16q_1q_2&16q_1q_2&4q_2^2 &0&0& 40q_1q_2^2 \\
0&0&0& 8q_1q_2&20q_1q_2&8q_1q_2 &0&0& 16q_1q_2q_0 \\
0&0&4q_2&0&0&0&8q_1q_2&4q_2^2&0 \\
0&0& 2q_1 &0&0&0&20q_1q_2&16q_1q_2&0 \\
0&2q_1&0&0&0&0&8q_1q_2&16q_1q_2&0 \\
0&0&0&0&2q_2&4q_2&0&0&8q_1q_2 \\
1&0&0&2q_1&0&0&0&0&8q_1q_2 \\
0&1&0&0&0&0&0&0&0 
\end{pmatrix}$$

\normalsize

$$M_{h_1^2h_2^2}=\small\begin{pmatrix}
0&0&0&32q_1^2q_2&16q_1q_2q_0&32q_1q_2^2&0&0&144q_1^2q_2^2 \\
0&8q_1q_2&8q_1q_2&0&0&0&16q_1q_2q_0&40q_1q_2^2&0 \\
0&8q_1q_2&8q_1q_2&0&0&0&40q_1^2q_2&16q_1q_2q_0&0 \\
0&0&0&8q_1q_2&8q_1q_2&4q_2^2&0&0&32q_1q_2^2 \\
0&0&0&8q_1q_2&12q_1q_2&8q_1q_2&0&0&16q_1q_2q_0 \\
0&0&0&4q_1^2&8q_1q_2&8q_1q_2&0&0&32q_1^2q_2 \\
0&0&2q_2&0&0&0&8q_1q_2&8q_1q_2&0 \\
0&2q_1&0&0&0&0&8q_1q_2&8q_1q_2&0 \\
1&0&0&0&0&0&0&0&0 
\end{pmatrix}$$

\normalsize

\medskip

\subsection{Comparison with cubics and Gushel--Mukai fourfolds}
In this Section, we use exactly the same notations as in \cite{jg}, except for the following Remark.

\begin{remark}
Given a variety $X$ a number field $K$ and an additional formal variable $\mathfrak{q}'$ of degree in $\{1;2\}$, the ring $\widehat{R}^*(X,K)$ in \cite{jg} is defined as
$$\widehat{R}^*(X,K) := K[{\mathfrak{q}'}^{\pm 1}][[Q,(T_k)_k]],$$
where $Q$ is the Novikov variable of $X$ and $(T_k)_k$ are formal variables attached to a given basis $(\alpha_k)_k$ of the Hodge locus $H^*(X)^\Hdg$.
Here, we rename it as $\widehat{R}^*(X,K)_\Hdg$ and we use the notation $\widehat{R}^*(X,K)_{(0)}$ for the corresponding ring where $(T_k)_k$ are formal variables attached to a given basis $(\alpha_k)_k$ of the degree-zero Hochschild cohomology $H^{(0)}(X)$.

Specifically, Property $\clubsuit_{\widehat{R}^*(X,K)_\Hdg}$ (see \cite[Definition 27]{jg}) is defined using evaluation maps $\ev \colon \widehat{R}^*(X,K)_\Hdg \to S^*_K$, whereas Property $\coeur_{\widehat{R}^*(X,K)_{(0)}}$ is actually defined using evaluation maps $\ev \colon \widehat{R}^*(X,K)_{(0)} \to S^*_K$.
This distinction is crucial in the following proof.
\end{remark}

\begin{prop}
\label{prop_properties}
Let $X$ be any Verra fourfold and $K$ be any number field.
Then $X$ satisfies Property $\clubsuit_{\widehat{R}^*(X,K)_\Hdg}$ and fails to satisfy Property $\coeur_{\widehat{R}^*(X,K)_{(0)}}$.
\end{prop}

\begin{proof}
For any $\zeta_1,\zeta_2 \in D(0,1)$, let $\ev_0 \colon \widehat{R}^*(X,\QQ)_{(0)} \to S^*$ be the $\QQ$-evaluation map defined by
$$\ev_0(\mathfrak{q}')=b^{\deg(\mathfrak{q}')}~,~~\ev_0(q_1)= \zeta_1 b^4~,~~\ev_0(q_2)= \zeta_2 b^4 ~,~~ \textrm{and} \quad \ev_0(T_k) = 0,$$
for each $0 \leq k \leq h$.

Propositions \ref{comput} and \ref{zero-antiinvariant} show that, for all values of $\zeta_1$ and $\zeta_2$, the generalized kernel $\mathbb{E}^X_{\ev_0,0}$ contains the primitive cohomology and a three-dimensional subspace of the ambient cohomology.
Hence, we obtain 
$$\nu^X_{\ev_0, 0}=1~,~~{\nu'}^X_{\ev_0, 0}=0~,~~\textrm{and}~~\gamma^X_{\ev_0, 0}=0$$
so that $X$ fails to satisfy Property $\coeur_{\widehat{R}^*(X,K)_{(0)}}$.

Assume for a moment that $X$ fails to satisfy Property $\clubsuit_{\widehat{R}^*(X,K)_\Hdg}$.
Then there exists an evaluation map $\ev_1 \colon \widehat{R}^*(X,K)_\Hdg \to S^*_K$ and an eigenvalue $\alpha_1 \in S^*_{K_{\ev_1}}$ satisfying
$$\nu^X_{\ev_1, \alpha_1} \neq 0~~\textrm{and}~~\rho^X_{\ev_1,\alpha_1}<3.$$
This evaluation map extends to an evaluation map $\ev_1 \colon \widehat{R}^*(X,K)_{(0)} \to S^*_K$ by sending the new variables $T_k$ (corresponding to classes of Hochschild degree zero that are not Hodge classes) to zero.
Since the dimension of the transcendental part of $H^*(X,\QQ)$ is $21$ and the dimension of the algebraic part is at most $2$, then we get
$$\mathrm{rk}_{S_{K_{\ev_1}}} (\mathbb{E}^X_{\ev_1,\alpha_1}) \leq  21+2<24.$$

Let now $X'$ be a rational Verra fourfold, whose existence is guaranteed by Proposition \ref{prop_rational_Verra}.
By \cite[proof of Theorem 56]{jg}, since $X'$ fails to satisfy Property $\coeur_{\widehat{R}^*(X,K)_{(0)}}$, then for any weak factorization of a birational morphism from $X'$ to $\PP^4$, one of the blow-up center must be a $K3$ surface $\Sigma$.
Furthermore, for any evaluation map $\ev \colon \widehat{R}^*(X',K)_{(0)} \to S^*_K$, there exists a unique $\alpha \in S^*_{K_\ev}$ such that
$$H^*(\Sigma, K_\ev)_{S_{K_\ev}} \hookrightarrow \mathbb{E}^{X'}_{\ev,\alpha}.$$
Note that this eigenvalue $\alpha$ is exactly the unique one that satisfies
$$\nu^{X'}_{\ev,\alpha} \neq 0.$$
Therefore $\alpha$ coincides with $\alpha_1$ when we consider the evaluation $\ev_1$. Since $H^{(0)}(X) = H^{(0)}(X')$ and since Gromov--Witten theory is deformation-invariant, we have
$$\mathrm{rk}_{S_{K_{\ev_1}}} (\mathbb{E}^X_{\ev_1,\alpha_1}) = \mathrm{rk}_{S_{K_{\ev_1}}} (\mathbb{E}^{X'}_{\ev_1,\alpha_1}) \geq 24,$$
yielding a contradiction.
Hence, $X$ must satisfy Property $\clubsuit_{\widehat{R}^*(X,K)_\Hdg}$.
\end{proof}

Recall that there are two derived equivalent generically, non-isomorphic degree two $K3$ surfaces $\Sigma_1$ and $\Sigma_2$ associated to $X$. They have isomorphic rational Hodge structures, which are also isomorphic to $H^4(X)^{\mathrm{prim}}$.

\begin{coro}
\label{rem_irr_atom}
Let $X$ be any Verra fourfold.
The $24$-dimensional generalized kernel found in the computation of Propositions \ref{comput} and \ref{zero-antiinvariant} is indecomposable, i.e.,  does not split when taking a general evaluation map $\ev \colon \widehat{R}^*(X,K)_{(0)} \to S^*_K$; it is an ‘‘irreducible'' atom of $X$. Furthermore, 
it is isomorphic, as a rational Hodge structure, to the rational cohomology of both $\Sigma_1$ and $\Sigma_2$.
\end{coro}

\begin{proof}
The first part directly follows from the proof of Proposition \ref{prop_properties}: by deformation invariance of GW invariants, one can assume the Verra to be rational, and then this generalized eigenspace is isomorphic to the corresponding one in the cohomology of a K3 surface. However, the quantum cohomology of a K3 has only one eigenvalue along any deformation, hence the first statement. For the second part, the result is just a consequence of the fact that all these rational Hodge structures are isomorphic to $\QQ^{\oplus 3}\oplus H^4(X)^{\mathrm{prim}}$. 
\end{proof}


\begin{coro}
Let $X$ be any Verra fourfold and $Y$ be either a very general cubic fourfold or a very general Gushel--Mukai fourfold.
Then $X$ is not birational to $Y$.
\end{coro}

\begin{proof}
The strategy is the same as for the irrationality of the very general cubic fourfold.
Assume that $X$ and $Y$ are birational and choose a weak factorisation between them. Since $X$ satisfies Property $\clubsuit_{\widehat{R}^*(X,K)_\Hdg}$, but $Y$ fails to satisfy Property $\clubsuit_{\widehat{R}^*(Y,K)_\Hdg}$, then there must be a blow-up center that fails to satisfy Property $\clubsuit$, which is impossible because candidates are either point, curves, or surfaces, and they all satisfy Property $\clubsuit$.
Consequently, $X$ and $Y$ are not birational.
\end{proof}

\begin{remark}
As a consequence, a very general Gushel--Mukai fourfold is not birational to a very general nodal Gushel--Mukai fourfold.
\end{remark}

\begin{theorem}\label{K3}
If $Y$ is a cubic fourfold or a Gushel-Mukai fourfold that is birational to some Verra fourfold $X$, then there exists a projective $K3$ surface $\Sigma$ and an isomorphism of rational Hodge structures
$$H^4(Y,\mathbb{Q})_\mathrm{\mathrm{van}} \simeq H^2(\Sigma,\mathbb{Q})(-1).$$
\end{theorem}

\begin{proof}
Consider a weak factorisation from $X$ to $Y$
$$X:=X_0 \leftrightarrow X_1 \leftrightarrow \dotsc, \leftrightarrow X_{r-1} \leftrightarrow X_r=:Y.$$

Consider the evaluation map used in the proof of \cite[Proposition 55]{jg} if $Y$ is a cubic fourfold,
or the evaluation map in the last remark of \cite{bmp} if $Y$ is a Gushel-Mukai. We denote it by $\ev_r$.

There is an eigenvalue $\alpha_r$ such that
$$\nu^{X_r}_{\ev_r, \alpha_r}\neq 0,~~{\nu'}^{X_r}_{\ev_r, \alpha_r}=0,~~\textrm{and}~~\gamma^{X_r}_{\ev_r, \alpha_r}=1.$$
We first show that the generalized eigenspace $\mathbb{E}^{X_r}_{\ev_r,\alpha_r}$ is isomorphic, as a Hodge structure over the field $F=\QQ((a^\QQ))$ (recall \cite[section 2]{jg}), to the full cohomology of a $K3$ surface.
By choosing general evaluation maps for the blow-up centers in the direction $\leftarrow$, we deduce evaluation maps $\ev_k$ on each $X_k$ (as well as evaluation maps for the blow-up centers in the direction $\rightarrow$).
Indeed, there are two possibilities:
\begin{enumerate}
    \item there exists an eigenvalue $\alpha_0$ such that $\mathbb{E}^{X_0}_{\ev_0,\alpha_0} \simeq \mathbb{E}^{X_r}_{\ev_r,\alpha_r}$, in which case $\mathbb{E}^{X_0}_{\ev_0,\alpha_0}$ is precisely the `irreducible atom' of $X$ that is isomorphic to the rational Hodge structure (tensored with $F$) of $\Sigma_1$ (or equivalently $\Sigma_2$),
    \item there exists an index $1 \leq k \leq r$ such that the blow-up center $Z_k \subset X_{k-1}$ of $X_{k-1} \leftarrow X_k$ contains a generalized eigenspace isomorphic to $\mathbb{E}^{X_r}_{\ev_r,\alpha_r}$, in which case $Z_k$ should be a surface whose minimal model is a $K3$ surface and this generalized eigenspace is precisely isomorphic to the rational Hodge structure (tensored with $F$) of that $K3$ surface.
\end{enumerate}
Then we observe that, e.g.~in the first case (and similarly in the second case), since the matrices $\ev_r(\kappa_\tau)$ and $\ev_0(\kappa_\tau)$ are conjugate to each other, then their ranks on $\mathbb{E}^{X_r}_{\ev_r,\alpha_r} \simeq \mathbb{E}^{X_0}_{\ev_0,\alpha_0}$ are equal, and thus both equal to one.
Therefore, using the same method as in the proof of \cite[Theorem 55]{jg}, we can get rid of the two copies of $F$ in $\mathbb{E}^{X_r}_{\ev_r,\alpha_r}$ and in $H^\bullet(K3,\mathbb{Q})_F$. Eventually, we conclude that
$$H^4(Y,\mathbb{Q})_\mathrm{van} \simeq H^2(\Sigma,\mathbb{Q})(-1), $$
for some K3 surface $\Sigma$, as claimed.
\end{proof}

\begin{remark}
If a suitable theory of atoms over $\ZZ$ was available, we should be able to improve the previous theorem and upgrade it to a statement about $\ZZ$-Hodge structures. Then, conditional to the Hassett conjecture, we would conclude that Verra fourfolds and cubic fourfolds (or Gushel-Mukai fourfolds) are birational only when they are both rational.
\end{remark}

\section{A first order deformation of the small quantum product}

In order to really understand the Hodge atoms of our Verra fourfold, we would need to control the behaviour of the (generalized) eigenspace $E_0$ along a general deformation of the small quantum product inside the big quantum cohomology ring.  
We will restrict ourselves tp the case of an infinitesimal deformation 
along  a class $\tau$, which will be enough to prove our second main result in the next subsection. In other words, we consider $(H^\bullet(X)[q_1,q_2][[t]], \star_t)$, where:
$$
\sigma_a\star_t \sigma_b=\sigma_a\cup \sigma_b + \sum_{\sigma_c,d_1,d_2,e}\langle \sigma_a,\sigma_b,\sigma_c^\vee,\tau^e \rangle_{d_1,d_2} \frac{q_1^{d_1}q_2^{d_2} t^e}{e!} \sigma_c.
$$
If $\tau=\sum_i \tau_i$ is a decomposition of $\tau$ in homogeneous elements, the Euler field in this deformation is  $$\mathrm{Eu}_{t\tau}=2h_1+2h_2+\sum_i\Big(1 - \frac{\deg(\tau_i)}{2}\Big)t\tau_i.$$
We compute  the matrix of the multiplication operator by $\mathrm{Eu}_{t \tau}$, at first order in $t$. 

\begin{lemma}
\label{lem_comp_def}
Suppose $\tau$ homogeneous. For $u\in H^\bullet(X)$ and  $d_1,d_2\geq 1$, the coefficient of $q_1^{d_1}q_2^{d_2} t^k$ in $\mathrm{Eu}_{t\tau}\star_t u$ is 
$$\frac{1}{k!}
 \Big( 2d_1+2d_2+1-\frac{\deg(\tau)}{2}\Big)
 \sum_{\sigma_a}\langle  u,\tau^k , \sigma_a^\vee \rangle_{d_1,d_2} \sigma_a.
$$
\end{lemma}

\begin{proof}
This follows from the definition and the divisor axiom. 
\end{proof}

\medskip
Consider the deformation of the quantum cohomology along the direction of   $$\tau=a_{11}h_1^2 +a_{12}h_1h_2+ a_{22}h_2^2+a_{112}h_1^2h_2+a_{122}h_1h_2^2+a_{1122}h_1^2h_2^2.$$  Then, along such a 
deformation, the quantum multiplication by $\mathrm{Eu}_{t\tau}$ at first order in $t$ is given by the following matrix (that we divided into four submatrices made of columns respectively 1,2,3, resp. 4,5 resp. 6,7, resp. 8,9):

{\small

$$
\begin{pmatrix}
0&4q_1(1+4q_2(a_{112}+a_{122})t)&4q_2(1+4q_2(a_{112}+a_{122})t)
& \dots \\
2&2q_1(a_{11}+4q_2a_{1122})t&2q_2(a_{12}+2a_{22}+4q_1a_{1122})t
& \dots \\
2&2q_1(2a_{11}+a_{12}+4q_2a_{1122})t& 2q_2(a_{22}+4a_{1122})t
& \dots \\
-a_{11}t&2&0
& \dots \\
-a_{12}t &2&2
& \dots \\
-a_{22}t&0&2
& \dots \\
-2a_{112}t&-a_{12}t&-(a_{11}+2a_{1122})t
& \dots \\
-2a_{122}t&-(a_{22}+2q_1a_{1122})t&-a_{12}t
& \dots \\
-3a_{1122}t&-2a_{122}t&-2a_{112}t
&  \dots 
\end{pmatrix}$$

$$\begin{pmatrix}
 \dots & 12q_1^2(a_{11}+8q_2a_{1122})t+24q_1q_2(a_{12}+a_{22})t &12q_1q_2(2a_{11}+3a_{12}+2a_{22}+4q_0a_{1122})t
 &  \dots \\
 \dots &4q_1(1+q_2(4a_{112}+8a_{122})t)&4q_2(1+q_1(10a_{112}+8a_{122})t)
 & \dots \\
 \dots &8q_1(1+q_1a_{112}t+2q_2a_{122}t)&4q_1(1+q_2(8a_{112}+10a_{122})t)
 &  \dots \\
 \dots &8q_1q_2a_{1122}t&2q_2(a_{12}+2a_{22}+4q_1a_{1122})t
 & \dots \\
 \dots &2q_1(2a_{11}+a_{12}+4q_2a_{1122})t&2q_1a_{11}t+2q_2a_{22}t+12q_1q_2a_{1122}t
 &  \dots \\
 \dots &4q_1(a_{12}+q_1a_{1122})t&2q_1(2a_{11}+a_{12}+4q_2a_{1122})t
 & \dots \\
 \dots &2&2
 & \dots \\
 \dots &0&2
 & \dots \\
 \dots &-a_{22}t&-a_{12}t
 &  \dots 
\end{pmatrix}$$

$$
\begin{pmatrix}
 \dots & 24q_1q_2(a_{11}+a_{12})t+12q_2^2(a_{22}+8a_{1122})t
 &32q_1q_2(1+5q_1a_{112}t+2q_0a_{122}t) &
 \dots \\
 \dots &8q_2(1+2q_1a_{112}t+q_2a_{122}t) & 12q_1q_2(2a_{11}+5a_{12}+2a_{22}+4q_0a_{1122})t 
 & 
 \dots \\
 \dots &4q_2(1+q_1(8a_{112}+4a_{112})t) & 12q_1^2(a_{11}+10q_2a_{1122})t+48q_1q_2(a_{12}+a_{22})t  & 
 \dots \\
 \dots  &4q_2(a_{12}+q_2a_{1122})t &4q_2(1+q_1(8a_{112}+4a_{122})t)&
 \dots \\
 \dots &2q_2(a_{12}+2a_{22}+4q_1a_{1122})t &4q_1(1+q_2(8a_{112}+10a_{122})t)& 
 \dots \\
 \dots &8q_1q_2a_{1122}t &8q_1(1+q_1a_{112}t+2q_2a_{122}t)& 
 \dots \\
 \dots &0 & 2q_2(a_{22}+4q_1a_{1122})t  & 
 \dots \\
 \dots &2 & 2q_1(2a_{11}+a_{12}+4q_2a_{1122})t  &
 \dots \\
 \dots &-a_{11}t &2& \dots 
\end{pmatrix}.$$

$$
\begin{pmatrix} 
 \dots &32q_1q_2(1+5q_2a_{122}t+2q_0a_{112}t)
  & 80q_1q_2(2q_1a_{11}+2q_2a_{22}+q_0a_{12}+9q_1q_2a_{1122})t \\
\dots & 48q_1q_2(a_{11}+a_{12})t+12q_2^2(a_{22}+10q_1a_{1122})t
 &32q_1q_2(1+2q_0a_{112}t+5q_2a_{122}t) \\
 \dots & 12q_1q_2(2a_{11}+5a_{12}+2a_{22}+4q_0a_{1122})t
 &32q_1q_2(1+5q_1a_{112}t+2q_0a_{122}t) \\
 \dots &8q_2(1+2q_1a_{112}t+q_2a_{122}t)
 & 24q_1q_2(a_{11}+a_{12})t+12q_2^2(a_{22}+8q_1a_{1122})t \\
 \dots &4q_2(1+q_1(10a_{112}+8a_{122})t)
 & 12q_1q_2(2a_{11}+5a_{12}+2a_{22}+4q_0a_{1122})t \\
\dots &4q_1(1+q_2(4a_{112}+8a_{122})t)
 & 12q_1^2(a_{11}+8q_2a_{1122})t+24q_1q_2(a_{12}+a_{22})t \\
 \dots & 2q_2(a_{12}+2a_{22}+4q_1a_{1122})t
 & 4q_2(1+4q_1(a_{112}+a_{122})t) \\
 \dots  & 2q_1(a_{11}+4q_2a_{1122})t
 &4q_1(1+4q_2(a_{112}+a_{122})t) \\
 \dots 
 &2& 0
\end{pmatrix}.$$
} \ 

\begin{prop}\label{deformation}
There exist deformations $\alpha(t), \beta(t), \gamma(t)$ of $\alpha,\beta,\gamma$ such that, at first order in $t$, $E_0^{\mathrm{amb}}(t):=\langle \alpha(t), \beta(t), \gamma(t) \rangle$ is preserved by $\mathrm{Eu}_{t\tau}\star_t$, with an action given in this basis by the matrix \vspace*{-3mm}
$$
t\begin{pmatrix}
 \lambda& 0 & 0\\
0 & \lambda  & 0\\
\mu & 0 &  \lambda
\end{pmatrix}, 
$$ \vspace*{-3mm}
$$ \mathit{where} \qquad\lambda= -2q_1a_{11}-2q_2a_{22}-4q_1q_2a_{1122}\quad \mathit{and} \quad \mu=a_{11}-a_{12}+a_{22}-2(q_1+q_2)a_{1122}.$$
\end{prop}

\begin{proof}
We just need to solve a linear system. We report the values of $\alpha(t), \beta(t), \gamma(t)$ which can be plugged into the matrix of the quantum multiplication by $\mathrm{Eu}_{t\tau}$, at first order in $t$, in order to obtain the result:
$$\begin{array}{rcl}
\alpha(t) &=& \alpha+t( a_{11}u_{11}+a_{12}u_{12}+a_{22}u_{22}+a_{112}u_{112}+a_{122}u_{122}+a_{1122}u_{1122})+O(t^2), \\
\beta(t) &=& \beta+t(a_{11}v_{11}+a_{12}v_{12}+a_{22}v_{22}+a_{112}v_{112}+a_{122}v_{122}+a_{1122}v_{1122})+O(t^2), \\
\gamma(t) &=& \gamma+ t(a_{11}w_{11}+a_{12}w_{12}+a_{22}w_{22}+a_{112}w_{112}+a_{122}w_{122}+a_{1122}w_{1122})+O(t^2), 
\end{array}$$

\noindent for the following classes: \small

$$
\begin{array}{rcl}
u_{11}=2q_1(h_1-h_2) ,\quad & \quad  u_{12}=2q_2h_1+2q_1h_2, \quad & \quad  u_{22}=2q_2(h_2-h_1),
\end{array}$$
$$\begin{array}{rcl}
&  u_{1122}=(-2q_2^2-2q_1q_2)h_1+(-2q_1^2-2q_1q_2)h_2+(q_1+q_2)(h_1^2h_2+h_1h_2^2), 
\end{array}$$
$$\begin{array}{rcl}
u_{112}=-4q_1(2q_2+q_1)+2(q_1+q_2)h_1^2 ,\quad & \quad  u_{122}=-4q_2(2q_1+q_2)+2(q_1+q_2)h_2^2;
\end{array}$$

$$\begin{array}{rcl}
v_{11}=-4q_1(2q_2+q_1)+2q_1h_1^2
,\quad & \quad  v_{12}=-2q_2h_1^2+2q_1h_2^2
,\quad & \quad v_{22}=+4q_2(2q_1+q_2)-2q_2h_2^2,
\end{array}$$
$$\begin{array}{rcl}
&v_{1122}=4q_1q_2(2q_1-2q_2-4h_1^2+4h_2^2),
\end{array}$$
$$\begin{array}{rcl}
v_{112}=4q_1q_2h_1-4q_1(2q_2+q_1)h_2 + 2q_1h_1^2h_2 ,\quad & \quad  v_{122}=-4q_1q_2h_2+4q_2(2q_1+q_2)h_1 - 2q_2h_2^2h_1;
\end{array}$$

$$\begin{array}{rcl}
w_{11}=2q_1(-4q_2h_1+2q_1h_2-h_1^2h_2) ,\quad &    w_{12}=-8q_1q_2(h_1+h_2) ,\quad &  w_{22}=2q_2(-4q_1h_2+2q_2h_1-h_2^2h_1), 
\end{array}$$
$$\begin{array}{rcl}
&  w_{1122}=-8q_1q_2^2h_1-8q_1^2q_2h_2-4q_1q_2(h_1^2h_2+h_1h_2^2), 
\end{array}$$
$$\begin{array}{rcl}
w_{112}=4q_1q_2(4q_2+2q_1-3h_1^2-2h_2^2) ,\quad & \quad  w_{122}=4q_1q_2(4q_1+2q_2-3h_2^2-2h_1^2).
\end{array}
$$

\normalsize\medskip\noindent
We give the Macaulay2 script used to get these formulas in the Appendix.
\end{proof}

\begin{remark} 
Since all classes in $H^\bullet(X)^{\mathrm{amb}}$ are multiples of the hyperplane classes, and since we have already obtained the full quantum Chevalley formula, an application of the divisor axiom and associativity of the quantum multiplication allows in principle to recover the full big ambient quantum multiplication along the deformation $\tau$. However, the actual computation of this multiplication at any order in $t$ seems out of reach. 
\end{remark}


\begin{lemma}
\label{lem_mult_id_antiinv}
The restriction of the multiplication by the Euler field to $H^\bullet(X)^{\mathrm{prim}}$ is equal to $\tilde{\lambda}(t)\mathrm{id}$, where
 $\tilde{\lambda}(t)=t\lambda+O(t^2)$.
\end{lemma}

\begin{proof}
The argument is similar to the one in the proof of Proposition \ref{zero-antiinvariant}. First notice that $\mathrm{Eu}_{t\tau}\star_t$ stabilizes $H^\bullet(X)^{\mathrm{prim}}$. Indeed, Gromov-Witten invariants with only one insertion in $H^\bullet(X)^{\mathrm{prim}}$ vanish by the irreducibility of the monodromy action on $H^\bullet(X)^{\mathrm{prim}}$. Then by Lemma \ref{lem_comp_def} and the fact that $\tau$ is ambient, the result follows. Then $H^\bullet(X)^{\mathrm{prim}}$ contains at least one eigenvector with a certain eigenvalue $\lambda'$; again by the irreducibility of the monodromy action, every (general hence all) element in $H^\bullet(X)^{\mathrm{prim}}$ is an eigenvector with the same eigenvalue. The result follows.
\end{proof}

By this Lemma and Remark \ref{rem_irr_atom},
$E^X_{t\tau}:=\langle \alpha(t),\beta(t),\gamma(t)\rangle \oplus H^\bullet(X)^{\mathrm{prim}}$ is the generalized eigenspace of the multiplication by the Euler field $\mathrm{Eu}_{t\tau}\star_t$ associated to the eigenvalue $\tilde{\lambda}(t)$. 

\begin{lemma}
$\mathrm{Eu}_{t\tau}\star_t(\langle \alpha(t),\beta(t),\gamma(t)\rangle)\subset \langle \alpha(t),\beta(t),\gamma(t)\rangle$.
\end{lemma}

\begin{proof}
This is once again a consequence of the irreducibility of the monodromy action on $H^\bullet(X)^{\mathrm{prim}}$, which implies that Gromov-Witten invariants with insertions with only one element in $H^\bullet(X)^{\mathrm{prim}}$ vanish. Hence, since $\tau$ is an ambient class, the image of an ambient class via $\mathrm{Eu}_{t\tau}\star_t$ is again an ambient class, and the result follows.
\end{proof}

As a consequence we obtain the following result.

\begin{prop}
\label{prop_def_matrix}
The matrix of the multiplication $\mathrm{Eu}_{t\tau}\star_t$ by the Euler vector field deformed along $\tau$ and restricted to $E^X_{t\tau}$ is given by
$$
\begin{pmatrix}
 M(t)& 0 \\
0 & \tilde{\lambda}(t)\mathrm{id}  \end{pmatrix}, 
$$
where $M(t)$ is the restriction to $\langle \alpha(t),\beta(t),\gamma(t)\rangle$, and $\tilde{\lambda}(t)\mathrm{id}$ is the restriction to $H^\bullet(X)^{\mathrm{prim}}$. The matrix $M(t)$ is generically of rank one, with only eigenvalue $\tilde{\lambda}(t)$, and coincides  at first order in $t$ with the matrix in Proposition \ref{deformation}.
\end{prop}

\section{Appendix}

Here is the Macaulay2 script we used in order to work in the small ambient quantum cohomology of a Verra fourfold. This ring  $S$ is constructed at the end of the program. It is obtained from the (small) quantum  Chevalley formula only (in the ambient cohomology), since all ambient classes are multiples of hyperplane classes. The degrees we use do not reflect the actual degrees: with our choices of degrees, the result of a product is expressed as a $\QQ[q_1,q_2]$-linear combination of $1,h_1,h_2,h_1^2,h_1h_2,h_2^2,h_1^2h_2,h_1h_2^2,h_1^2h_2^2$.

\begin{verbatim}
R=QQ[q1,q2,h1,h2, h11, h12, h22, h112, h122, h1122,
Degrees=>{1,1,10,10,10,10,10,10,10,10}];
R1=h1*h1-h11-2*q1;
R2=h1*h2-h12;
R3=h2*h2-h22-2*q2;
R4=h11*h1-2*q1*h1-4*q1*h2;
R5=h12*h1-h112-2*q1*h2;
R6=h22*h1-h122;
R7=h22*h2-2*q2*h2-4*q2*h1;
R8=h12*h2-h122-2*q2*h1;
R9=h11*h2-h112;
R10=h112*h1-2*q1*h12-4*q1*h22-8*q1*q2;
R11=h122*h1-h1122-2*q1*h22-8*q1*q2;
R12=h1122*h1-2*q1*h122-8*q1*q2*h1-8*q1*q2*h2;
R13=h122*h2-2*q2*h12-4*q2*h11-8*q1*q2;
R14=h112*h2-h1122-2*q2*h11-8*q1*q2;
R15=h1122*h1-2*q1*h122-8*q1*q2*h1-8*q1*q2*h2;
S=R/ideal(R1,R2,R3,R4,R5,R6,R7,R8,R9,R10,R11,R12,R13,R14);
\end{verbatim}

\end{document}